\documentclass{article}
\usepackage[utf8]{inputenc}
\usepackage{xy}
\xyoption{all}
\usepackage{amssymb}
\usepackage{amsmath}
\usepackage{amsthm}
\usepackage{mathtools}
\usepackage{hyperref}
\usepackage{cleveref}
\usepackage{xcolor}
\usepackage{comment}
\usepackage{txfonts} % For \nequiv
\usepackage{tikz,tikz-cd}
\usepackage{url}
\usepackage[bottom=2cm, right=2cm, left=2cm, top=2cm]{geometry}

\newtheorem{theorem}{Theorem}[section]
\newtheorem{proposition}[theorem]{Proposition}
\newtheorem{lemma}[theorem]{Lemma}
\newtheorem{definition}[theorem]{Definition}

\theoremstyle{definition}
\newtheorem{remark}[theorem]{Remark}

\newcounter{lettered}
\newtheorem{letteredtheorem}[lettered]{Theorem}

\newcommand{\Z}{\mathbb{Z}}
\newcommand{\Q}{\mathbb{Q}}

\DeclareMathOperator{\hocolim}{{\rm hocolim}}
\DeclareMathOperator{\fib}{{\rm fib}}
\DeclareMathOperator{\pt}{{\rm pt.}}
\DeclareMathOperator{\tors}{{\rm tors}}
\DeclareMathOperator{\torsfree}{{\rm torsfree}}

\newcommand{\Zp}{\hat{\Z}_{p}}
\newcommand{\Qp}{\hat{\Q}_{p}}
\newcommand{\Ab}{{\rm Ab}}
\newcommand{\ob}{{\rm ob}}
\newcommand{\Fin}{{\rm Fin}}
\DeclareMathOperator{\Ext}{{\rm Ext}}
\DeclareMathOperator{\Spec}{{\rm Spec}}
\newcommand{\Zpx}{\Zp^{\times}}

\DeclareMathOperator{\im}{{\rm im\:}}

\newcommand{\Sp}{S_{p}^{\wedge}}
\newcommand{\KUp}{KU_p^{\wedge}}

\newcommand{\Gal}{{\rm Gal}}

\newcommand{\red}[1]{\textcolor{red}{(#1)}}
\newcommand{\gal}[1]{\operatorname{Gal}(#1)}

\title{Iwasawa invariants of finite spectra}
\author{Austin Maison and Andrew Salch}
\date{\today}

\begin{document}

\maketitle
\abstract{We calculate the classical Iwasawa invariants of the Iwasawa modules associated to the $p$-adic topological $K$-theory of finite spectra. We show that the graded average of the orders of $n$ consecutive $K(1)$-local homotopy groups of a finite spectrum $X$ grows asymptotically like $\frac{-\log_p(n)}{2}$ times the total Iwasawa $\lambda$-invariant of $X$. We show that the Iwasawa $\mu$-invariants of finite spectra are all zero. Finally, we prove a spectral analogue of a weak form of the Iwasawa Main Conjecture, describing the orders of the $K(1)$-local homotopy groups of a certain ``torsion-free replacement'' of $X$ in terms of the characteristic polynomials of the Iwasawa modules associated to $X$.}

%\tableofcontents

\section{Introduction}

\subsection{The point of this paper}

Throughout this document, we will let $p$ be an odd prime.

In classical Iwasawa theory, an {\em Iwasawa module} is a module over the group algebra $\hat{\mathbb{Z}}_p[\hat{\mathbb{Z}}_p]$. 
Iwasawa modules arise naturally in the study of towers of Galois extensions of number fields 
\begin{align}
\label{tower 1}    F = F_0 \subseteq F_1 \subseteq F_2 \subseteq \cdots \subseteq F_{\infty} = \bigcup_n F_n
\end{align}
such that, for each $n$, the Galois group $\gal{F_n/F}$ is isomorphic to $\mathbb{Z}/p^n\mathbb{Z}$. 
Iwasawa's structure theory for Iwasawa modules defines three important algebraic invariants of an Iwasawa module: the $\lambda$-invariant, the $\mu$-invariant, and the characteristic polynomial. The $\lambda$- and $\mu$-invariants are integers, while the characteristic polynomial is an element of the polynomial ring $\Zp[T]$. 

Iwasawa showed that the class groups of the number fields in \eqref{tower 1} naturally fit together to form an Iwasawa module whose $\lambda$- and $\mu$-invariants satisfy the equality 
\begin{align} \label{hfn eq 1} h_{F_n}^{(p)} &= p^{n\lambda + p^n \mu + \nu}\end{align}
for some integer $\nu$ and for all sufficiently large $n$. Here $h_{F_n}^{(p)}$ is the $p$-part of the class number of $F_n$. The point is that, if we consider how the $p$-part of the class group grows as we move up the tower \eqref{tower 1}, the asymptotic rate of growth is controlled by the $\lambda$- and $\mu$-invariants. A fuller introduction to the $\lambda$- and $\mu$-invariants is given in \cref{The Classical Iwasawa Main Conjecture}.

Meanwhile, the characteristic polynomial of the Iwasawa module is used to define an {\em algebraic} $p$-adic $L$-function of the tower \eqref{tower 1}. In \cref{background section} we give a fuller exposition of the number-theoretic background, including the {\em analytic} $p$-adic $L$-function and the Iwasawa Main Conjecture, i.e., the relationship between the algebraic and analytic $p$-adic $L$-functions.

Iwasawa modules also arise in stable homotopy theory, although they have not received nearly as much study in stable homotopy theory as in number theory. Via the stable Adams operations, the profinite group $\Zpx$ acts on the $p$-adic $K$-theory of any spectrum\footnote{For readers with a background in number theory but not in stable homotopy theory, we remark that the main theorems in this paper concern {\em finite} spectra. The suspension spectrum of every finite pointed CW-complex is a finite spectrum, and conversely, every finite spectrum is a finite desuspension of the suspension spectrum of a finite CW-complex. Our point is that, for a reader who is unfamiliar with spectra, it is harmless to think of finite spectra as simply {\em the stable homotopy types of finite pointed CW-complexes}, as well as formal desuspensions thereof. In particular, all the theorems in this paper about $K(1)$-local homotopy groups of finite spectra are also theorems about $K(1)$-local homotopy groups of finite pointed CW-complexes. }. Using the isomorphism of profinite groups $\Zpx \cong \Zp\times \mathbb{F}_p^{\times}$, one can canonically split the $p$-adic $K$-theory of any spectrum $X$ into $2(p-1)$ summands \[\epsilon_0(\KUp)^0(X),\ \epsilon_1(\KUp)^0(X),\ \dots,\epsilon_{p-2}(\KUp)^0(X),\ \epsilon_0(\KUp)^{-1}(X),\ \epsilon_1(\KUp)^{-1}(X),\ \dots,\epsilon_{p-2}(\KUp)^{-1}(X),\] each of which is an Iwasawa module. This splitting was already given in 1969 by Adams in \cite{MR0251716}, although to our knowledge the first published mention of Iwasawa theory in connection with this splitting was not until Ravenel's paper \cite{MR737778} in 1984. Steve Mitchell wrote several papers \cite{MR1367298},\cite{MR2181837} about how the Iwasawa-theoretic arguments used by number theorists could be seen in terms of algebraic $K$-theory spectra, but the only published paper about Iwasawa theory {\em of spectra themselves}---that is, the only paper which applies Iwasawa theory to draw conclusions about topology, not about number theory---is the 2007 paper \cite{MR2327028} of Rebekah Hahn and Steve Mitchell. Hahn and Mitchell focused on category-theoretic results, e.g. a classification of the thick subcategories of the homotopy category of weakly dualizable $L_{K(1)}S^0$-module spectra, given in terms of subsets of $\Spec \Zp[\Zpx]$. We review some of Hahn--Mitchell's work in \cref{Review of Hahn--Mitchell's results}.

By contrast, the aim of this paper is computational, rather than category-theoretic: we calculate the Iwasawa invariants of all finite spectra, express them in topological terms, and we prove topological analogues of well-known Iwasawa-theoretic results in number theory. We will now sketch the three main results in this paper---one for each of the classical Iwasawa invariants ($\lambda$-invariant, $\mu$-invariant, and characteristic polynomial)---together with their number-theoretic precedents.

\subsection{The $\lambda$-invariant}

The formula \eqref{hfn eq 1}, above, shows the precise sense in which the $\lambda$-invariant (together with the $\mu$-invariant, but this vanishes, as we explain below) controls the asymptotic rate of growth of the $p$-part of the class number in an appropriate tower of number fields. 
Our topological analogue, stated below in Theorem \ref{intro thm on lambda and euler char}, is that {\em graded average of the orders of $n$ consecutive $K(1)$-local homotopy groups of a spectrum grows asymptotically like $\frac{-\log_p(n)}{2}$ times the total $\lambda$-invariant of the spectrum.} To be clear, throughout this paper, $\log_p$ denotes the ordinary base $p$ logarithm, not any kind of $p$-adic logarithm. 
``Graded averages'' are explained in \cref{Asymptotic graded averages.}, but the idea is very simple: instead of  $\frac{1}{n}\sum_{j=1}^n \left| \pi_j(L_{K(1)}X)\right|$, the average of the orders of the first $n$ $K(1)$-local homotopy groups of $X$, we count the odd-degree homotopy groups as {\em negative} and the even-degree homotopy groups as {\em positive}. This yields the alternating sum $\frac{1}{n}\sum_{j=1}^n (-1)^j\left| \pi_j(L_{K(1)}X)\right|$. In our asymptotic analysis, we also allow skipping finitely many initial homotopy groups, i.e., we consider the asymptotic properties of the alternating sum $\frac{1}{n}\sum_{j=1+m}^{n+m} (-1)^j\left| \pi_j(L_{K(1)}X)\right|$ as both $m$ and $n$ go to infinity.
\begin{letteredtheorem}[Theorem \ref{lambda and euler char}]\label{intro thm on lambda and euler char}
For all finite spectra $X$, the limit
\begin{equation*}
\lim_{m\rightarrow\infty} \lim_{n\rightarrow\infty} \frac{\frac{1}{n}\sum_{j=1+m}^{n+m} (-1)^j \left| \pi_jL_{K(1)}X\right|}{-\lambda(X)\cdot \log_p(n)/2}
\end{equation*}
is equal to $1$.
\end{letteredtheorem}
Here $\lambda(X)$ is the {\em total $\lambda$-invariant of $X$}, defined in Definition \ref{def of tot lambda} as the alternating sum of the $\lambda$-invariants of the Iwasawa modules of $X$:
\begin{align*} \lambda(X) &= 
 \sum_{j=0}^{p-2}\left( \lambda(\epsilon_j(\KUp)^0(X)) - \lambda(\epsilon_j(\KUp)^{-1}(X))\right) .\end{align*}

We remark that our proof of Theorem \ref{intro thm on lambda and euler char} involves a bit of algebraic $K$-theory. In Proposition \ref{additivity of sn}, we prove that, in a general triangulated category $\mathcal{C}$ satisfying appropriate hypotheses, the asymptotic growth rate of the graded average of the orders of $n$ consecutive homotopy groups is an {\em additive} invariant. Hence this asymptotic growth rate depends only on an object's equivalence class in the Grothendieck group $K_0(\mathcal{C})$. Since $K_0$ of the category of finite spectra is isomorphic to $\mathbb{Z}$ and generated by $[S^0]$, we are able to prove Theorem \ref{intro thm on lambda and euler char} by reducing to the case where $X$ is a sphere. Another consequence of the algebraic-$K$-theoretic approach is that it is easy to see that the total $\lambda$-invariant of a finite spectrum is also equal to its Euler characteristic (see Lemma \ref{additivity of lambda}).

As far as we know, the idea of treating asymptotic growth of orders of homotopy groups as an additive invariant, and using algebraic $K$-theoretic methods to study it, is new. The method is general and flexible enough to be applied more broadly, e.g. at higher heights using the noncommutative Iwasawa theory of Venjakob \cite{MR1924402}. The authors are pursuing this approach at height $2$ at primes $p>3$, but this goes beyond the scope of this paper. More generally, throughout this paper we have attempted to use methods and ideas which appear to be generalizable to higher heights, at sufficiently large primes, using Venjakob's theory.

Aside from the total $\lambda$-invariant of a finite spectrum $X$, one can also consider the $\lambda$-invariants of each of the individual Iwasawa modules $\epsilon_j(\KUp)^0(X)$ and $\epsilon_j(\KUp)^{-1}(X)$. These are not as exciting as the total $\lambda$-invariant: in the proof of Lemma \ref{additivity of lambda}, we show that the individual $\lambda$-invariants are described topologically by the equalities
\begin{align*}
 \lambda(\epsilon_j(\KUp)^0(X)) 
  &= \sum_{n\equiv j \mod p-1} \dim_{\mathbb{Q}}H_{2n}(X), \mbox{\ \ \ and} \\
 \lambda(\epsilon_j(\KUp)^{-1}(X)) 
  &= \sum_{n\equiv j \mod p-1} \dim_{\mathbb{Q}}H_{2n-1}(X) .
\end{align*}

\subsection{The $\mu$-invariant} 

Iwasawa conjectured that, in the case of a tower \eqref{tower 1} obtained from a number field $F$ by iteratively adjoining a primitive $p$th root of unity, $p^2$th root of unity, $p^3$th root of unity, etc., the $\mu$-invariant of the resulting Iwasawa module is zero. The celebrated theorem of Ferrero--Washington \cite{MR0528968} proved this conjecture in the case where $F$ is an abelian extension of $\Q$. Our point is that, in number theory, one often expects the $\mu$-invariant to vanish.

In Theorem \ref{intro mu-invs vanish}, we prove that this expectation is borne out for finite spectra: 
\begin{letteredtheorem}[Theorem \ref{mu-invs vanish}]\label{intro mu-invs vanish}
Let $X$ be a finite spectrum. Then, for every $i,j$, the $\mu$-invariant of each of the Iwasawa modules $\epsilon_j(\KUp)^i(X)$ is zero. 
\end{letteredtheorem}

\subsection{The characteristic polynomial}

Hahn--Mitchell define a {\em pseudo-equivalence} to be a map of spectra $X \rightarrow Y$ whose induced map in $p$-adic $K$-theory $(\KUp)^i(Y) \rightarrow (\KUp)^i(X)$ has finite kernel and finite cokernel for each $i$. In Theorem \ref{pseudoequiv thm} we prove that every finite spectrum $X$ is pseudo-equivalent to a canonical-up-to-homotopy wedge product $\overset{\circ}{X} = (X_{even})_{torsfree}\vee (X_{odd})_{torsfree}$ such that
\begin{itemize}
\item $(X_{even})_{torsfree}$ is a finite spectrum such that $H_*((X_{even})_{torsfree};\mathbb{Z}_{(p)})$ is torsion-free and concentrated in even degrees, and
\item $(X_{odd})_{torsfree}$ is a finite spectrum such that $H_*((X_{odd})_{torsfree};\mathbb{Z}_{(p)})$ is torsion-free and concentrated in odd degrees.
\end{itemize}
Theorem \ref{pseudoequiv thm}
refines a theorem of Hahn and Mitchell which we recall below as Theorem \ref{hahn-mitchell thm 2}. The spectrum $\overset{\circ}{X}$ appears in the statement of Theorem \ref{intro imc 2}, below, which establishes that the characteristic polynomials of the Iwasawa modules of a finite spectrum carry the information of the orders of the $K(1)$-local homotopy groups of $\overset{\circ}{X}$.

Before stating Theorem \ref{intro imc 2}, we state its number-theoretic precedent. One consequence of the classical Iwasawa Main Conjecture, proven by Mazur and Wiles \cite{MR0742853}, is that {\em the $p$-adic valuations of special values of the characteristic polynomial of the Iwasawa module of \eqref{tower 1} agree with the $p$-adic valuations of special values of a classical Dirichlet $L$-function.} We call this fact---stated more precisely in \cref{The Classical Iwasawa Main Conjecture}---the {\em weak form of the Iwasawa Main Conjecture.}

Theorem \ref{intro imc 2} is a spectral analogue of the weak form of the Iwasawa Main Conjecture. 
Rather than the characteristic polynomial of the Iwasawa module of a tower of number fields \eqref{tower 1}, Theorem \ref{intro imc 2} concerns values of the characteristic polynomial $f^{\epsilon_j(\KUp)^iX}(T)$ of the Iwasawa module $\epsilon_j(\KUp)^i(X)$ of $X$, for $j = 0, \dots ,p-2$ and for $i\in \{ -1,0\}$. The role of the special values of a Dirichlet $L$-function is played by the orders of the $K(1)$-local homotopy groups of $\overset{\circ}{X}$. It is reasonable to think that the orders of these homotopy groups are indeed analogous to special values of a Dirichlet $L$-function, since in \cite{salch2023kulocal}, the orders of those homotopy groups are shown to have, up to sign, the same $p$-adic valuations as the special values of a certain complex-analytic ``$KU$-local zeta-function of $X$'' which is a product of Tate twists of Dirichlet $L$-functions.

The notation $x\sim_p y$ in the statement of Theorem \ref{intro imc 2} denotes that $x$ and $y$ have the same $p$-adic valuation.
\begin{letteredtheorem}[Theorem \ref{imc 2}]\label{intro imc 2}
    Let $X$ be a finite spectrum.  Let $\alpha$ and $\beta$ be the least and greatest integers $i$, respectively, such that $H_{i}(X;\mathbb{Q})$ is nontrivial. Then we have
\begin{align*}
 \left|\pi_{2m-1}L_{K(1)}D\overset{\circ}{X}\right| 
  &\sim_p f^{\epsilon_{-m}(\KUp)^0(X)}((1+p)^{-m}-1) \mbox{\ \ and}\\
 \left|\pi_{2m}L_{K(1)}D\overset{\circ}{X}\right| 
  &\sim_p f^{\epsilon_{-m}(\KUp)^{-1}(X)}((1+p)^{-m}-1) 
\end{align*}
for all integers $m$ satisfying $m < \frac{-\beta}{2}$, and also for all integers $m$ satisfying $m > \frac{-\alpha}{2}$.
\end{letteredtheorem}
The subscript $-m$ of $\epsilon_{-m}$, in the statement of Theorem \ref{intro imc 2}, is to be understood as being defined modulo $p-1$.

The proof of Theorem \ref{intro imc 2} resembles the proof of \cite[Theorem 2.8]{salch2023kulocal}, and more generally, the material on characteristic polynomials in this paper can be seen as $p$-adic versions of results on complex-analytic ``$KU$-local zeta-functions'' of spectra given in \cite{salch2023kulocal}. These $p$-adic versions are, in some ways, stronger and more general: for example, Theorem \ref{intro imc 2} applies to all finite spectra, while \cite[Theorem 2.8]{salch2023kulocal} required the relevant spectrum to have trivial homology in all odd degrees.

\subsection{Conventions}\label{conventions section}
\begin{itemize}
\item $p$ will always be an {\em odd} prime.
\item Given rational numbers $a,b$, we write $a\sim_p b$ to denote that $a$ and $b$ have the same $p$-adic valuation.
\end{itemize}

\subsection{Acknowledgments}
We thank Francesc Castella for graciously visiting Wayne State University and answering many of our questions about Iwasawa theory.

\section{Background ideas from Iwasawa theory}
\label{background section}
\subsection{Review of classical Iwasawa theory}
\label{The Classical Iwasawa Main Conjecture}

In this section, we will take a brisk walk through classical Iwasawa theory, arriving at a statement of the Iwasawa Main Conjecture. Our treatment here is based primarily on \cite{MR3840072}, \cite{MR1421575}, and \cite{sharifi}, but it is by no means an exhaustive exposition. Instead, we will simply highlight the objects and tools that will be used to applied to the Iwasawa modules of finite spectra in the subsequent sections.

Consider a tower of Galois extensions of the form \eqref{tower 1}. 
Write $\Gamma_n$ for the Galois group $\gal{F_n/F}$, and $\Gamma$ for the limit $\lim_n \Gamma_n$ of the Galois groups. 
Iwasawa studied such towers in the case that $\Gamma_n \cong \Z/p^n\Z$ and $\Gamma \cong \Zp$. Under such circumstances, the tower \eqref{tower 1} is called a \emph{$\Zp$-extension}. Iwasawa was interested in determining the asymptotic growth rate of the class number $h_{F_n} = \left|\operatorname{Cl}_{F_n}\right|$ as $n$ increases. Iwasawa focused in particular on the growth rate of the order of the $p$-Sylow subgroup $A_n$ of $\operatorname{Cl}_{F_n}$. We will write $h_{F_n}^{(p)}$ for the order of this $p$-Sylow subgroup. Iwasawa proved the following result:
\begin{theorem}[Iwasawa] \label{Iwasawa's theorem}
    There exist nonnegative integers $\lambda$ and $\mu$ and an integer $\nu$ such that
    \begin{align*}
        h_{F_n}^{(p)} = p^{n\lambda + p^n \mu + \nu}
    \end{align*}
    for all sufficiently large $n$.
\end{theorem}
Iwasawa conjectured that $\mu=0$ in the case of the \emph{cyclotomic $\Zp$-extension}, i.e., the unique $\Zp$-extension $F_{\infty}/F$ inside $\bigcup_n F(\zeta_{p^n})$. This conjecture was proven for $F$ an abelian extension of $\Q$ by Ferrero and Washington \cite{MR0528968}.

For each $\Zp$-extension, let $X_{\infty}$ be the limit $\lim_{n} A_n$ of the sequence of norm maps 
\[ \dots \stackrel{N}{\longrightarrow} A_{2} \stackrel{N}{\longrightarrow} A_1\stackrel{N}{\longrightarrow} A_0 .\]
 Since each $A_n$ is a $\Zp[\Gamma_n]$-module, $X_{\infty}$ is a module over the \emph{Iwasawa algebra} $\Lambda = \Zp[[\Gamma]]$. Modules over the ring $\Lambda$ are called \emph{Iwasawa modules}, and there are at least two reasons why Iwasawa modules are central in classical Iwasawa theory. First, as observed by Serre, $\Lambda$ is isomorphic to a power series ring with a single generator
\begin{align*}
    \Lambda = \Zp[[\Gamma]] \cong \Zp[[T]]
\end{align*}
where a choice of topological generator $\gamma \in \Gamma$ maps to $1+T \in \Zp[[T]]$. We will fluidly use both descriptions of $\Lambda$ throughout this paper. Second, finitely generated Iwasawa modules have a nice structure theorem, similar to finitely generated modules over a principal ideal domain, but the structure theory describes the module up to pseudo-isomorphism rather than up to isomorphism. A map of Iwasawa modules is called a {\em pseudo-isomorphism} if its kernel and cokernel are both finite. Given Iwasawa modules $M,N$, we write $M\simeq N$ if there exists a pseudo-isomorphism $M \to N$.\footnote{It is worth noting that $M \simeq N$ does not imply $N \simeq M$ unless $M,N$ are both finitely-generated $\Lambda$-torsion Iwasawa modules.} %finite zig-zag of pseudo-isomorphisms between $M$ and $N$. 
\begin{theorem}[Serre, Iwasawa]
    For any finitely generated Iwasawa module $M$, we have
    \begin{align}
\label{iwasawa decomp}        M &\simeq \Lambda^r \oplus \bigoplus_{i=1}^s \Lambda/p^{m_i} \oplus \bigoplus_{j=1}^t \Lambda/f_j(T)^{n_j},
    \end{align}
    where each $f_j(T) \in \Zp[T]$ is an irreducible, distinguished polynomial, that is, an irreducible, monic polynomial such that $f(T) \equiv T^{\operatorname{deg}(f)} \mod{p}$.
\end{theorem}
The invariants $r, m_i, n_j, f_j$ of $M$ are unique up to permutation of the sequence $(m_1, \dots ,m_s)$ and simultaneous permutation of the sequences $(n_1, \dots ,n_t)$ and $(f_1(T), \dots ,f_t(T))$. The nonnegative integer $r$ is, of course, the {\em rank} of $M$. 
The rank of an Iwasawa module is a pseudo-isomorphism invariant.
When $M$ is also $\Lambda$-torsion (i.e., of rank zero), three additional pseudo-isomorphism invariants of $M$ play special roles in classical Iwasawa theory:
\begin{itemize}
    \item The {\em $\mu$-invariant} of $M$ is $\mu(M) = \sum_{i=1}^s m_i$.

    \item The {\em $\lambda$-invariant} of $M$ is $\lambda(M) = \sum_{j=1}^t n_j \operatorname{deg}(f_j)$.

    \item The {\em characteristic polynomial} of $M$ is $f^M(T) = p^{\mu(M)} \prod_{j=1}^t f_j(T)^{n_j}$.
\end{itemize}
It turns out that $X_{\infty}$ is $\Lambda$-torsion, and $\mu(X_{\infty})$ and $\lambda(X_{\infty})$ agree with the numbers $\mu$ and $\lambda$ from \Cref{Iwasawa's theorem}. In other words, the structure of $X_{\infty}$ determines the asymptotic growth rates of the $p$-part of class groups in a $\Zp$-extension. %We will not return to the $\lambda$-invariant or the $\mu$-invariant until \cref{Iwasawa invariants of finite spectra}, but the characteristic polynomial will play a role in virtually everything that follows.

For the remainder of this section, we will focus our attention on the cyclotomic $\Zp$-extension in which $F_n=\Q(\zeta_{p^n})$ for each $n$, and its corresponding Iwasawa module $X_{\infty}$. In this case we have $\Gamma = \lim_{n\rightarrow\infty}\gal{\mathbb{Q}(\zeta_{p^n})/\mathbb{Q}(\zeta_p)}$. As a consequence of our running assumption that $p$ is an {\em odd} prime, the profinite group $\hat{\mathbb{Z}}_p^{\times}$ of $p$-adic units splits as a product $\hat{\mathbb{Z}}_p\times \mathbb{F}_p^{\times}$. In terms of Galois groups, this splitting takes the form $\Gamma^{\prime} \cong \Gamma\times \Delta$, where:
\begin{itemize}
\item $\Gamma^{\prime} = \gal{F_{\infty}/\Q} \cong \Zpx$,
\item and $\Delta = \gal{F/\Q} \cong \mathbb{F}_p^{\times}$.
\end{itemize}
We furthermore define $\Lambda^{\prime}$ as $\Zp[[\Gamma']]$. We have an isomorphism of topological $\Zp$-algebras $\Lambda^{\prime} \cong \Lambda[\Delta]$. Using the product decomposition $\Gamma' \cong \Gamma \times \Delta$, it is natural to think of $\Gamma$ as the group of principal units $1+p\Zp\cong \Zp$, and of $\Delta$ as the $(p-1)$st roots of unity, both sitting inside $\Zpx$.

The theory of Iwasawa modules can be extended very simply to $\Lambda'$-modules in the following way. There exist idempotent elements $\epsilon_0,\epsilon_1, \dots ,\epsilon_{p-2} \in \Zp[\Delta]$ of the form
\begin{align*}
    \epsilon_j = \frac{1}{p-1} \sum_{a=1}^{p-1} \omega^{-j}(\sigma_a) \sigma_a.
\end{align*}
For $a=1,2,\dots, p-1$, $\sigma_a \in \Delta$ sends $\zeta_p$ to $\zeta_p^a$, and $\omega:\Delta \to \Zpx$ denotes the {\em Teichm\"{u}ller character}, given by letting $\omega(\sigma_a)$ be the unique $(p-1)$st root of unity in $\Zp$ such that $\omega(\sigma_a) \equiv a \mod{p}$. 

The idempotents $\epsilon_0, \dots ,\epsilon_{p-2}$ yield a splitting $\Lambda^{\prime}\cong \epsilon_0\Lambda^{\prime} \times \dots \times \epsilon_{p-2}\Lambda^{\prime}$, in which each factor $\epsilon_{j}\Lambda^{\prime}$ is isomorphic, as a topological ring, to $\Lambda$. Hence, for any $\Lambda'$-module $M$, we have a canonical splitting $M \cong \bigoplus_{j=0}^{p-2}\epsilon_j M$ where each eigenspace $\epsilon_j M$ is an Iwasawa module. Since $X_{\infty}$ is, in fact, a $\Lambda'$-module, we can consider its eigenspaces $\epsilon_j X_{\infty}$ and the associated characteristic polynomials $f^{\epsilon_j X_{\infty}}(T)$. 

In \cite{MR0163900}, Kubota and Leopoldt analytically constructed $p$-adic $L$-functions which interpolate the special values of Dirichlet $L$-functions. Specifically, if $\chi$ is a Dirichlet character, then there exists a $p$-adic $L$-function $L_p(s,\chi)$, with $s$ taking values in the domain $\Zp$, such that
\begin{align*}
    L_p(1-n,\chi) = (1-\chi\omega^{-n}(p)p^{n-1})\cdot L(1-n,\chi\omega^{-n})
\end{align*}
for all $n \geq 1$. The essential point here is that the Kubota--Leopoldt $p$-adic $L$-function is {\em constructed by means of $p$-adic analysis}, and {\em it is constructed in such a way that its special values are known in advance to be meaningful.}

Iwasawa \cite{MR0360526} constructed $p$-adic $L$-functions of Dirichlet characters by an entirely different method, exploiting the isomorphism $\Zp[[\Gamma]] \cong \Zp[[T]]$ to great effect. Iwasawa constructed a power series $f(T, \omega^{1-j}) \in \Zp[[T]]$ such that, after a change of variable,
\begin{align*}
    f((1+p)^s-1,\omega^{1-j}) = L_p(s,\omega^{1-j})
\end{align*}
and he famously made the following conjecture which was later proved by Mazur and Wiles \cite{MR0742853}.
\begin{theorem}[Iwasawa Main Conjecture]
    For $j=3,5,\dots, p-2$, the following is an equality of ideals in $\Lambda$.
    \begin{align}
\label{imc}        (f^{\epsilon_j X_{\infty}}(T)) &= (f(T,\omega^{1-j}))
    \end{align}
\end{theorem}
The left-hand side of \eqref{imc} is an {\em algebraic} $p$-adic $L$-function, i.e., a $p$-adic $L$-function extracted in a simple way from an Iwasawa module, but whose special values at integers are {\em a priori} of uncertain significance. The right-hand side of \eqref{imc} is an {\em analytic} $p$-adic $L$-function (like those constructed by Kubota--Leopoldt), with {\em a priori} meaningful special values at negative integers, but of completely uncertain connection to algebraic constructions like cohomology or Iwasawa modules.
The Iwasawa Main Conjecture asserts that, up to multiplication by a unit in the power series ring $\hat{\mathbb{Z}}_p[[T]]$, the algebraic $p$-adic $L$-function coincides with the analytic $p$-adic $L$-function.

The equality \eqref{imc} yields that
\begin{align*}
    f^{\epsilon_j X_{\infty}}((1+p)^{1-n}-1) &\sim_p L(1-n,\omega^{1-j-n})
\end{align*}
for all $n \geq 1$. The moral is that, as a consequence of the Iwasawa Main Conjecture, 
\begin{equation}
\label{weak main conj}
\parbox{6in}{\em{The $p$-adic valuations of the special values of the characteristic polynomial $f^{\epsilon_j X_{\infty}}$ agree with the $p$-adic valuations of the special values of a classical Dirichlet $L$-function.}}
\end{equation}
We will refer to statement \eqref{weak main conj} as the {\em weak form of the Main Conjecture}. It is this weak form for which we will be able to formulate a spectral analogue, in Theorem \ref{imc 2}.

\subsection{The $p$-adic complex $K$-theory of a spectrum yields a $2(p-1)$-tuple of Iwasawa modules}
\label{tuples section}

Let $KU$ denote the periodic complex $K$-theory spectrum, and $\KUp$ its $p$-adic completion. For a spectrum $X$, the splitting of the $\Lambda^{\prime}$-module $(\KUp)^0(X)$ as a direct sum
 $\epsilon_0(\KUp)^0(X)\oplus \dots \oplus\epsilon_{p-2}(\KUp)^0(X)$ is well-known in topology, dating back to Adams; our understanding is that \cite{MR0251716} is the original reference. This splitting is described in explicitly Iwasawa-theoretic terms in \cite{MR1367298} and in \cite{MR2327028}. 
Similarly, in odd degrees, we have the $(p-1)$-tuple of Iwasawa modules $\epsilon_0(\KUp)^{-1}(X)\oplus \dots \oplus\epsilon_{p-2}(\KUp)^{-1}(X)$. By Bott periodicity, we do not need to work with the $p$-adic $K$-groups in all degrees: it suffices to consider the $2(p-1)$-tuple of Iwasawa modules coming from any single even degree and any single odd degree. In this paper we choose to work with $(\KUp)^0$ and $(\KUp)^{-1}$.

\subsection{Review of Hahn--Mitchell's results}
\label{Review of Hahn--Mitchell's results}

In \cite{MR2327028}, Rebekah Hahn and Steve Mitchell developed some basic ideas and results in spectral Iwasawa theory. We review the most relevant ideas from their paper.
\begin{definition}\label{pseudo-iso of spectra}
A map of $K(1)$-local spectra $X\rightarrow Y$ is a {\em pseudo-equivalence} if the induced maps in $p$-adic $K$-theory $(\KUp)^0(Y) \rightarrow (\KUp)^0(X)$ and $(\KUp)^{-1}(Y) \rightarrow (\KUp)^{-1}(X)$ each have finite kernel and finite cokernel. 
\end{definition}
Hence a map of $K(1)$-local spectra $X\rightarrow Y$ is a pseudo-equivalence if and only if its induced map of Iwasawa modules $\epsilon_j(\KUp)^i(Y) \rightarrow \epsilon_j(\KUp)^i(X)$ is a pseudo-isomorphism for all $i,j$---or equivalently, a pseudo-isomorphism for all $i\in \{0,-1\}$ and all $j\in \{0,1, \dots ,p-2\}$.

\begin{definition}\label{def of elementary spectrum}
An Iwasawa module $M$ is {\em elementary cyclic} if $M$ is isomorphic to either $\Lambda$, or $\Lambda/p^i$ for some positive integer $i$, or $\Lambda/f^i$ for some positive integer $i$ and some irreducible distinguished polynomial $f$. 

An Iwasawa module $M$ is {\em elementary} if it is a direct sum of elementary cyclic Iwasawa modules.

A $K(1)$-local spectrum $X$ is {\em elementary} if the Iwasawa module $\epsilon_j(\KUp)^i(X)$ is elementary for all $i,j$.
\end{definition}

Recall that we have a single fixed odd prime $p$ throughout this paper. Consequently it is safe to use the terms defined in Definitions \ref{pseudo-iso of spectra} and \ref{def of elementary spectrum} more broadly, by saying that:
\begin{itemize}
\item a spectrum $X$ (not necessarily $K(1)$-local) is {\em elementary} if its $K(1)$-localization is elementary.
\item A map of (not necessarily $K(1)$-local) spectra $X \rightarrow Y$ is a {\em pseudo-equivalence} if the induced maps $(\KUp)^0(Y) \rightarrow (\KUp)^0(X)$ and $(\KUp)^{-1}(Y) \rightarrow (\KUp)^{-1}(X)$ each have finite kernel and finite cokernel.
\item Spectra $X,Y$ are {\em pseudo-equivalent} if there is a zigzag of pseudo-equivalences connecting $X$ and $Y$.
\end{itemize}

\begin{theorem}(\cite[Theorem 9.3]{MR2327028})\label{hahn-mitchell thm 1}
Let $X$ be a $K(1)$-local spectrum such that the $\Lambda^{\prime}$-modules $(\KUp)^0(X)$ and $(\KUp)^{-1}(X)$ are each finitely generated. Then there exist $K(1)$-local spectra $Y$ and $Z$, and pseudo-equivalences
\[ X \stackrel{\sim}{\longrightarrow} Y \stackrel{\sim}{\longleftarrow} Z ,\]
with $Z$ elementary.
\end{theorem}

One has to be a bit careful when writing something like ``$X$ and $Y$ are pseudo-equivalent,'' since if there exists a pseudo-equivalence $X\rightarrow Y$, there does not necessarily exist a pseudo-equivalence $Y\rightarrow X$.

Under slightly stronger hypotheses, there is a convenient particular case of Theorem \ref{hahn-mitchell thm 1} in which Hahn--Mitchell get a slightly stronger result:
\begin{theorem}(\cite[Theorem 9.6]{MR2327028})\label{hahn-mitchell thm 2}
Let $X$ be a $K(1)$-local spectrum such that the $\Lambda^{\prime}$-modules $(\KUp)^0(X)$ and $(\KUp)^{-1}(X)$ are each finitely generated, and such that each of the $\Lambda$-modules $\epsilon_j(\KUp)^i(X)$ is $\Lambda$-torsion for all $i,j$.
 Then there exists an elementary $K(1)$-local spectrum $\tilde{X}$ and pseudo-equivalences
\begin{align*} X \stackrel{\sim}{\longrightarrow} \tilde{X} &\ \ \ \ \mbox{and}\  \tilde{X} \stackrel{\sim}{\longrightarrow} X .\end{align*}
\end{theorem}
We will refer to $\tilde{X}$ as the {\em Hahn--Mitchell replacement} for $X$.

In Theorem \ref{hahn-mitchell thm 2}, the phrase ``for all $i,j$'' can be safely replaced by ``for all $i\in \{0,-1\}$ and all $j\in \{0,1, \dots ,p-2\}$''.

\section{The characteristic polynomials and Iwasawa $\mu$-invariants of finite spectra}
\label{Section SIMC}

\subsection{$\KUp$ of the Spheres}
\label{KUp of the Spheres}

As a warm-up exercise, we let $X$ be an even-dimensional sphere $S^{2i}$, and we consider the Iwasawa-theoretic properties of the Iwasawa modules $\epsilon_0(\KUp)^*(X),\dots,\epsilon_{p-2}(\KUp)^*(X)$. We have the isomorphism $(\KUp)^0(S^{2i}) \cong \Zp$ for all $i \in \Z$, and $\Gamma' \cong \Zpx$ acts on $(\KUp)^0(S^{2i})$ via the stable Adams operations, i.e. for all $u \in \Zpx$ and all $x \in (\KUp)^0(S^{2i})$,
\begin{align*}
    u \cdot x &= u^ix.
\end{align*}
Among the $p-1$ eigenspaces $\epsilon_j (\KUp)^0(S^{2i})$ of $(\KUp)^0(S^{2i})$, it turns out that only one is nontrivial:
\begin{proposition}
    $\epsilon_j (\KUp)^0(S^{2i}) \cong 
    \begin{cases}
        (\KUp)^0(S^{2i}) & \text{if } i \equiv j \mod{p-1}\\
        0 & \text{otherwise}
    \end{cases}$
\end{proposition}
\begin{proof}
    We can think of $\Delta$ as sitting inside $\Zpx$ via the Teichmüller character which means that $\sigma_a \in \Delta$ acts on $(\KUp)^0(S^{2i})$ by multiplication by $\omega^i(\sigma_a)$. So, for $1 \in (\KUp)^0(S^{2i})$, we have
    \begin{align*}
        \epsilon_j \cdot 1 &= \frac{1}{p-1} \sum_{a=1}^{p-1} \omega^{-j}(\sigma_a) \sigma_a \cdot 1\\
%        &= \frac{1}{p-1} \sum_{a=1}^{p-1} \omega^{-j}(\sigma_a) \omega^i(\sigma_a)\\
        &= \frac{1}{p-1} \sum_{a=1}^{p-1} \omega^{i-j}(\sigma_a) \\
%    \end{align*}
%    and since $\omega(\sigma_a)$ is a $p-1$-st root of unity in $\Zp$, we have 
%    \begin{align*}
%        \epsilon_j \cdot 1 
 &= 
        \begin{cases}
            1 & \text{if } i \equiv j \mod{p-1}\\
            0 & \text{otherwise}.
        \end{cases}
    \end{align*}
\end{proof}
This mean that $\epsilon_j (\KUp)^0(S^{2i})$ is a finitely generated $\Lambda$-torsion Iwasawa module such that $\mu(\epsilon_j (\KUp)^0(S^{2i})) = 0$ for all $j$, and it is nontrival when $j \equiv i \mod{p-1}$. We now restrict our attention to the nontrivial case $j\equiv i \mod{p-1}$. The characteristic polynomial of an Iwasawa module essentially records the action of a topological generator in $\Gamma \cong 1+p\Zp$, so we can construct the characteristic polynomial of $\epsilon_j (\KUp)^0(S^{2i})$, which we will denote as $f_{i,j}(T)$ for ease of notation, in the following way. Recall that there exists an isomorphism $\Zp[[1+p\Zp]] \cong \Zp[[T]]$ which is defined by sending the topological generator $1+p \in 1+p\Zp$ to $1+T \in \Zp[[T]]$. %So, $1+T$ acts in the same way as $1+p$, that is, by multiplication by $(1+p)^i$. 
Hence, in $\Zp[[T]]$, we identify $1+T$ and $(1+p)^i$ and conclude that the characteristic polynomial of $\epsilon_j (\KUp)^0(S^{2i})$ is
\begin{align*}
    f_{i,j}(T) &= 
    \begin{cases}
        T-(1+p)^i+1 & \text{if } i \equiv j \mod{p-1} \\
        1 & \text{otherwise}
    \end{cases}    
\end{align*}
Note that this implies that $\epsilon_j (\KUp)^0(S^{2i}) \simeq \Lambda/f_{i,j}(T)$
and $\lambda(\epsilon_j (\KUp)^0(S^{2i})) = 1$ in the nontrivial case. Otherwise, the $\lambda$-invariant is $0$.

In the same vein as the weak form of the Main Conjecture, \eqref{weak main conj}, the $p$-adic valuations of the special values of the characteristic polynomials $f_{i,j}(T)$ agree with the orders of the stable homotopy groups of the $K(1)$-local sphere. Specifically, we get the following relationship.

%\red{I commented out the last statement of prop 3.2 because only one of the $p-1$ char polys will recover the nontrivial homotopy groups of $L_{K(1)}S^0$. In fact, it recovers all but the $0$th. The others will "recover" the trivial homotopy groups since they themselves are trivial. -AM}
\begin{comment}
\begin{proposition} \label{SIMC for Spheres}
    For each of the characteristic polynomials $f_{i,0}(T),f_{i,1}(T),\dots ,f_{i,p-2}(T)$ of the Iwasawa module of $\epsilon_j (\KUp)^0(S^{2i})$ \red{I think this could be restated since we're talking about $p-1$ Iwasawa modules. See the comment in the tex file below this for one suggestion.}, we have
    \begin{align*}
        f_{i,j}((1+p)^{1-n}-1) &\sim_p \left|\pi_{2(n+i-1)-1} L_{K(1)}S^0\right|
    \end{align*}
    for all %$i \equiv j \mod{p-1}$ and all 
 $n \equiv 1 - j \mod{p-1}$.
\end{proposition}
\end{comment}

\begin{proposition} \label{SIMC for Spheres}
    For the characteristic polynomial $f_{i,j}$ of the Iwasawa module $\epsilon_j (\KUp)^0(S^{2i})$, where $j=0,1,\dots,p-2$, we have
    \begin{align*}
        f_{i,j}((1+p)^{1-n}-1) &\sim_p \left|\pi_{2(n+i-1)-1} L_{K(1)}S^0\right|
    \end{align*}
    for all $n \equiv 1 - j \mod{p-1}$.
\end{proposition}
We do not write out a proof of Proposition \ref{SIMC for Spheres}, since it merely amounts to a rephrasing, in Iwasawa-theoretic terms, of the classical calculation (essentially \cite[Theorem 8.10]{MR737778})
\begin{align*}
        \pi_t L_{K(1)} S^0 = 
        \begin{cases}
            \Zp & \text{if } t=-1,0\\
            \Z/p^{k+1}\Z & \text{if } t=2(p-1)p^k r-1, r \not\equiv 0 \mod{p}\\
            0 & \text{otherwise}
        \end{cases}
    \end{align*}

The stable homotopy groups of $L_{K(1)}S^0$ are trivial in degrees $\nequiv -1$ modulo $2(p-1)$, with the one exception of $\pi_0(L_{K(1)}S^0)$, which is isomorphic to $\Zp$. 
Hence Proposition \ref{SIMC for Spheres} recover the orders of all the nontrivial stable homotopy groups of $L_{K(1)}S^0$ other than $\pi_0 L_{K(1)}S^0$. 

There is one special case of Proposition \ref{SIMC for Spheres} which requires a bit of explanation:
if $n=1-i$, then \ref{SIMC for Spheres} makes a claim about the order of $\pi_{-1}(L_{K(1)}S^0)$, which is the infinite group $\Zp$. 
Since $\Zp = \lim_m \Z/p^m\Z$, we formally treat the $p$-adic valuation of $\left|\Zp\right|$ as the limit
\begin{align*} 
    \nu_p\left(\left|\Zp\right|\right) &= \lim_m \nu_p\left(\left|\Z/p^m\Z\right|\right) = \infty,
\end{align*}
which is consistent with the claim $0 \sim_p \left|\Zp\right|$ made by Proposition \ref{SIMC for Spheres}.

\begin{remark}
    The $\Lambda^{\prime}$-module $(\KUp)^0(S^{2i})$ also arises number-theoretically. Consider the cyclotomic $\Zp$-extension $F_{\infty}/F$ where $F_n = \Q(\zeta_{p^n})$. If we let $R_n$ denote the group of roots of unity sitting inside the $p$-adic completion of $\mathcal{O}_{F_n}^{\times}$, then $R_n = \langle \zeta_{p^n} \rangle \cong \Z/p^n\Z$. We can arrange these groups into a tower by considering the norm maps $N:R_{n+1} \to R_n$. The reader can check that $N(\zeta_{p^{n+1}})=\zeta_{p^n}$. Hence we have that $R= \lim_n R_n \cong \Zp$. Each $(\Z/p^n\Z)^{\times} \cong \gal{F_n/\Q}$ acts on $\Z/p^n\Z \cong R_n$ by left multiplication, and, in the limit, $\Zpx \cong \lim_n \gal{F_n/\Q}$ acts on $\Zp \cong R$ by left multiplication. Therefore, we get the $\Lambda'$-module isomorphism $R \cong (\KUp)^0(S^2)$.

    Suppose we ``twist" the Galois action of $(\Z/p^n\Z)^{\times}$ on $\Z/p^n\Z \cong R_n$ so that, for $a \in (\Z/p^n\Z)^{\times}$ and $x \in \Z/p^n\Z$, $a \cdot x = a^i x$. Denote the resulting $\Lambda'$-module as $R(i)$. Then $R(i) \cong (\KUp)^0(S^{2i})$.
\end{remark}

\subsection{The $\mu$-invariants of finite spectra.}

There is little to say about $\mu$-invariants of finite spectra: they are all zero, for easy reasons.
\begin{theorem}\label{mu-invs vanish}
Let $X$ be a finite spectrum. Then, for every $i,j$, the $\mu$-invariant of each of the Iwasawa modules $\epsilon_j(\KUp)^i(X)$ is zero. 
\end{theorem}
\begin{proof}
If the $\mu$-invariant of $\epsilon_j(\KUp)^i(X)$ were nonzero, then $\epsilon_j(\KUp)^i(X)$ would be pseudo-isomorphic to an Iwasawa module with $\Lambda/p^k \cong \mathbb{Z}/p^k\mathbb{Z}[[T]]$ as a summand, for some positive integer $k$. Hence $\Lambda/p^k$, an Iwasawa module with uncountably many linearly independent $p$-torsion elements, would occur as a summand in that Iwasawa module. A pseudo-isomorphism could only eliminate finitely many of these elements, after which $\epsilon_j(\KUp)^i(X)\subset (\KUp)^i(X)$ would still have uncountably many linearly independent $p$-torsion elements. This is impossible for a finite spectrum. Hence the $\mu$-invariant of $\epsilon_j(\KUp)^i(X)$ must have been zero.
\end{proof}

\subsection{The characteristic polynomials and Iwasawa main conjecture for finite spectra}
\label{Finite Spectra with Torsion-Free Homology}

We begin with an easy lemma about pseudo-equivalences.
\begin{lemma}\label{pseudoequivs lemma 1}\leavevmode
\begin{itemize}
\item
A map of spectra $f: X\rightarrow Y$ with homotopy fiber $\fib f$ is a pseudo-equivalence if and only if $(\KUp)^i(\fib f)$ is finite for all $i$.
\item If $X,Y$ are finite spectra and $f: X\rightarrow Y$ is a pseudo-equivalence, then $f$ induces an isomorphism in rational homotopy.
\end{itemize}
\end{lemma}
\begin{proof}\leavevmode\begin{itemize}
\item
Elementary from the long exact sequence induced in $p$-adic $K$-theory by the homotopy fiber sequence $\fib f \rightarrow X \stackrel{f}{\longrightarrow} Y$.
\item 
Suppose, by contrapositive, that $H\mathbb{Q}\wedge f : H\mathbb{Q}\wedge X \rightarrow H\mathbb{Q}\wedge Y$ is not an equivalence. Then the Spanier-Whitehead dual of $f$, $Df$, would also have the property that $H\mathbb{Q}\wedge Df$ is not an equivalence.
Hence $\KUp\wedge H\mathbb{Q}\wedge Df$ would also not be an equivalence, since $\KUp\wedge H\mathbb{Q}$ splits as a wedge of copies of suspensions of $H\mathbb{Q}$. Now by the chain of isomorphisms
\begin{align}
\label{iso 0901}    (\KUp)^i(X) \otimes_{\Z} \Q &\cong (\KUp)^i(X)[p^{-1}]\\
\nonumber    &\cong \operatorname{colim}_p (\KUp)^i(X)\\
\nonumber    &\cong\pi_{-i} (\operatorname{hocolim}_p F(X, \KUp))\\
\nonumber    &\cong \pi_{-i} (F(X, \operatorname{hocolim}_p \KUp))\\
\nonumber    &\cong \pi_{-i} (F(X, H\Q \wedge \KUp))\\
\nonumber    &\cong \pi_{-i} (F(X,S^0) \wedge H\Q \wedge \KUp)\\
\label{iso 0909}    &\cong (\KUp)_{-i}(DX \wedge H\Q),
\end{align}
the map $f$ does not induce an isomorphism in rationalized $p$-adic $K$-theory. Hence the map induced by $f$ in $p$-adic $K$-theory must have infinite kernel or infinite cokernel in at least one degree, i.e., $f$ is not a pseudo-equivalence.
\end{itemize}
\end{proof}

\begin{theorem}\label{pseudoequiv thm}
Let $X$ be a finite spectrum. Then %(TRUE, BUT NOT EXACTLY THE WAY THE PROOF IS PHRASED) there exists a pseudo-equivalence $X \rightarrow \overline{X}$ where $\overline{X}$ is a finite $S_{(p)}$-module spectrum with torsion-free homology. 
$X$ is pseudo-equivalent to a canonical-up-to-homotopy wedge sum 
\[ (X_{even})_{torsfree}\vee (X_{odd})_{torsfree}\] of finite spectra, such that $H_*\left((X_{even})_{torsfree};\mathbb{Z}_{(p)}\right)$ is torsion-free and concentrated in even degrees, and $H_*\left((X_{odd})_{torsfree};\mathbb{Z}_{(p)}\right)$ is torsion-free and concentrated in odd degrees.
\end{theorem}
\begin{proof}
A size argument analogous to that used to prove Theorem \ref{mu-invs vanish} shows that, for all $i,j$, the Iwasawa module $\epsilon_j(\KUp)^i(X)$ is finitely generated, and must have rank zero, hence must also be torsion. 
Theorem \ref{hahn-mitchell thm 2} then tells us that $X$ is pseudo-equivalent to its Hahn--Mitchell replacement $\tilde{X}$, whose $p$-adic $K$-theory admits an isomorphism 
\begin{align*}
 (\KUp)^i(\tilde{X}) 
  &\cong \oplus_{j=0}^{p-2}\epsilon_j(\KUp)^i(\tilde{X})
\end{align*}
for each $i$, such that $\epsilon_j(\KUp)^i(\tilde{X})$ is a direct sum of finitely many $\Lambda$-modules of the form $\Lambda/f(T)^m$ with $f(T)$ an irreducible distinguished polynomial, and $m$ a positive integer. Our point is simply that $(\KUp)^*(\tilde{X})$ is $p$-torsion-free.

Hahn--Mitchell use Bousfield's algebraicity results from \cite{MR0796907} to construct $\tilde{X}$. Among other results in \cite{MR0796907}, Bousfield shows that every graded $(KU_{(p)})_*(KU_{(p)})$-comodule $M$ splits into the direct sum of a summand $M_{even}$ concentrated in even degrees, and a summand $M_{odd}$ concentrated in odd degrees, and for each of these summands, if its injective dimension in the category of $(KU_{(p)})_*(KU_{(p)})$-comodules is $\leq 1$, then {\em that summand} is realizable as the $(KU_{(p)})$-homology of a spectrum, and this spectrum is unique up to $KU_{(p)}$-local equivalence. Bousfield also proves a finiteness result \cite[Theorem 9.7]{MR0796907}: if $M_{even}$ is finitely generated in each degree as a $\mathbb{Z}_{(p)}$-module, then one can find a {\em finite} spectrum whose $KU_{(p)}$-homology is $M_{even}$; and similarly for $M_{odd}$. Analogous results hold for $\KUp$, and are used by Hahn--Mitchell to prove the existence of $\tilde{X}$. There is no difficulty with switching between $p$-adic $K$-homology and $p$-adic $K$-cohomology in this argument, since $X$ is finite, so we may take Spanier-Whitehead duals to exchange the $K$-homology and $K$-cohomology whenever convenient. Consequently $\tilde{X}$ splits as a wedge $\tilde{X}_{even}\vee \tilde{X}_{odd}$,  with $\tilde{X}_{even}$ and $\tilde{X}_{odd}$ each finite spectra, with the $p$-adic $K$-theory of $\tilde{X}$ concentrated in even degrees, and with the $p$-adic $K$-theory of $\tilde{X}$ concentrated in odd degrees.

By Theorem \ref{hahn-mitchell thm 2}, $\tilde{X}$ has torsion-free $p$-adic $K$-theory, and hence $\tilde{X}_{even}$ and $\tilde{X}_{odd}$ each have torsion-free $p$-adic $K$-theory. We are not quite done, since we do not know that $\tilde{X}_{even}$ and $\tilde{X}_{odd}$ have torsion-free $p$-local {\em homology}. 
Torsion-free homology implies torsion-free $K$-theory, but the converse is not true: the Smith-Toda complex $V(1)$ is a counterexample.

We will now modify $\tilde{X}_{even}$ and $\tilde{X}_{odd}$ each by a pseudo-equivalence to produce finite spectra whose $p$-local homology is torsion-free. %\footnote{One could also try to bring Bousfield localization theory to bear on our situation, attempting to build a model structure on spectra whose acyclic objects are those spectra whose $p$-adic $K$-theory is finite in each degree. One hopes that the localization of $X$ in the resulting model structure agrees with the $K(1)$-localization of the desired spectrum $\overline{X}$. But from localization theory alone, we know of no means of showing that $\overline{X}$ is finite.}. 
Here is our construction, which uses only very classical tools. 
Let $Y$ be a finite spectrum. 
Then there exists a finite sequence of maps of spectra
\begin{equation}\label{homology decomp}\xymatrix{ \pt \ar[r] & Y^a \ar[r] & Y^{a+1} \ar[r] & \dots \ar[r] & Y^b}\end{equation}
such that $Y^b = Y$ and such that the cofiber of the map $Y^{j-1}\rightarrow Y^{j}$ is a $p$-local Moore spectrum\footnote{By a {\em $p$-local Moore spectrum of type $(A,n)$}, we mean a spectrum $M$ such that $H_j(M,\mathbb{Z}_{(p)})$ is trivial if $j\neq n$, and is isomorphic to $A$ if $j=n$.} of type $(H_j(Y;\mathbb{Z}_{(p)}),j)$. Such a sequence is called a ``homology decomposition,'' and in the unstable case, dates back to the 1959 papers \cite{MR0110096} and \cite{MR0108791}; the method used in those papers also works to produce the stable construction in spectra. 

Now we begin an inductive construction. We start in the bottom degree of the homology decomposition \eqref{homology decomp}: write $Y((a))$ for $Y$ with the torsion summand $M(\tors H_a(Y;\mathbb{Z}_{(p)}),a)$ pinched off. That is, $Y((a))$ is the homotopy cofiber of the composite map $M(\tors H_a(Y;\mathbb{Z}_{(p)}),a) \rightarrow Y^a \rightarrow Y$. The map $Y\rightarrow Y((a))$ is a pseudo-equivalence by the first part of Lemma \ref{pseudoequivs lemma 1}, and the homology $H_*(Y((a));\mathbb{Z}_{(p)})$ is torsion-free in degrees $\leq a$.

That was the first step in an induction. Here is the inductive hypothesis: suppose that $j$ is some integer, and suppose we have already constructed a 
%zigzag of pseudo-equivalences \[ \xymatrix{ Y \ar[d] & Y((a))^{\prime} \ar[ld]\ar[d] & \dots \ar[ld] & Y((n))^{\prime} \ar[d] Y((a)) & Y((a+1)) & \dots Y((n))}\]
sequence of pseudo-equivalences of finite spectra 
\[ Y \rightarrow Y((a)) \rightarrow Y((a+1)) \rightarrow \dots \rightarrow Y((j))\]
such that, for each $i\leq j$, the homology $H_*(Y((i));\mathbb{Z}_{(p)})$ is torsion-free in degrees $\leq i$. 
Choose a homology decomposition 
\begin{equation}\label{homology decomp 2} \dots \rightarrow Y((j))^{j-1}\rightarrow Y((j))^{j} \rightarrow Y((j))^{j+1} \rightarrow \dots\end{equation}
for $Y((j))$. 
The homology group $H_{j+1}(Y((j));\mathbb{Z}_{(p)})$ decomposes, as a $\mathbb{Z}_{(p)}$-module, into the direct sum of a torsion summand and a torsion-free summand. This induces a wedge splitting of the Moore spectrum $M = Y((j))^{j+1}/Y((j))^{j}$ as \[ M(\tors H_{j+1}(Y;\mathbb{Z}_{(p)}),j+1)\vee M(\torsfree H_{j+1}(Y;\mathbb{Z}_{(p)}),j+1) .\]
Write $f$ for the composite map
\[
\Sigma^{-1} M(\tors H_{j+1}(Y;\mathbb{Z}_{(p)}),j+1) \hookrightarrow 
\Sigma^{-1} \left(M(\tors H_{j+1}(Y;\mathbb{Z}_{(p)}),j+1)\vee M(\torsfree H_{j+1}(Y;\mathbb{Z}_{(p)}),j+1) \rightarrow \right) \rightarrow Y((j))^{j} \rightarrow Y((j)).
\]
Let $Y((j+1))$ denote the homotopy cofiber of the map $f$.  
Then, from analysis of the long exact sequences induced in homology by the homotopy cofiber sequence 
\[ \Sigma^{-1} M(\tors H_{j+1}(Y;\mathbb{Z}_{(p)}),j+1) \rightarrow Y((j)) \rightarrow Y((j+1)),\]
we get that the $p$-local homology groups of $Y((j+1))^{\prime}$ agree with those of $Y((j))$ in degrees $\neq j+1$, hence are torsion-free in degrees $\leq j$, and furthermore $H_{j+1}(Y((j+1));\mathbb{Z}_{(p)})$ is torsion-free. The map $Y((j)) \rightarrow Y((j+1))$ is a pseudo-equivalence by the first part of Lemma \ref{pseudoequivs lemma 1}. This completes the inductive step.

Since $Y$ is finite, its homology is trivial except in finitely many degrees, so after finitely many steps in the induction, we have killed off all the torsion in $H_*(-;\mathbb{Z}_{(p)})$. Hence the homotopy colimit $\hocolim_j Y((j))$ is in fact reached after finitely many steps. We define $Y_{torsfree}$ to be that homotopy colimit, so that $Y\rightarrow Y_{torsfree}$ is a pseudo-equivalence.
At each stage in the induction, we have attached finitely many cells to $Y$, hence $Y_{torsfree}$ is finite, as desired.

There is one last claim to explain: by construction, $(X_{even})_{torsfree}$ has torsion-free $p$-local homology, and has $p$-adic $K$-theory concentrated in even degrees. 
We need to show that $H_*\left( (X_{even})_{torsfree};\mathbb{Z}_{(p)}\right)$ is also concentrated in even degrees. 
By a theorem of Dold (see 14.18 of \cite{MR0198464} or Corollary 2.6 of \cite{MR0254838}), the shortest nonzero differential in the Atiyah-Hirzebruch spectral sequence $H_*((X_{even})_{torsfree}; (\KUp)_*)\Rightarrow (\KUp)_*((X_{even})_{torsfree})$ is torsion-valued. Since $(\KUp)_*$ is torsion-free and $(X_{even})_{torsfree}$ has torsion-free $p$-local homology, this Atiyah-Hirzebruch spectral sequence collapses immediately. Hence, if $(X_{even})_{torsfree}$ had any nonvanishing $p$-local homology in odd degrees, it would also have nonvanishing $p$-adic $K$-homology in some odd degree. This is impossible by the following argument: by the universal coefficient sequence for the $p$-adic $K$-theory of a finite CW-complex (\cite{anderson}, \cite{MR0388375}, \cite[Theorem IV.4.5]{MR1417719}):
\[
 0 
  \rightarrow \Ext^1_{\hat{\mathbb{Z}}_p}\left( (\KUp)_{n-1}\left((X_{even})_{torsfree}\right),\hat{\mathbb{Z}}_p\right) 
  \rightarrow (\KUp)^n\left((X_{even})_{torsfree}\right)
  \rightarrow \hom_{\hat{\mathbb{Z}}_p}\left( (\KUp)_n\left((X_{even})_{torsfree}\right),\hat{\mathbb{Z}}_p\right) 
  \rightarrow 0,  
\]
 the torsion-freeness of $(\KUp)^*\left( (X_{even})_{torsfree}\right)$ and the vanishing of $(\KUp)^*\left((X_{even})_{torsfree}\right)$ in odd degrees implies that\linebreak $(\KUp)_*\left( (X_{even})_{torsfree}\right)$ also vanishes in odd degrees. We conclude that $(X_{even})_{torsfree}$ must have had trivial $p$-local homology in all odd degrees. A completely similar argument shows that $(X_{odd})_{torsfree}$ must have had trivial $p$-local homology in all even degrees. 
\end{proof}
\begin{definition}
Given a finite spectrum $X$, we write $\overset{\circ}{X}$ as an abbreviation for the spectrum $(X_{even})_{torsfree} \vee (X_{odd})_{torsfree}$ constructed in Theorem \ref{pseudoequiv thm}. We call $\overset{\circ}{X}$ the {\em torsion-free replacement of $X$}.
\end{definition}

\begin{theorem}\label{iwasawa invs and rational type}
Let $X$ be a finite spectrum. Then, for each $i,j$, the $\lambda$-invariant, $\mu$-invariant, and characteristic polynomial of the Iwasawa module $\epsilon_j(\KUp)^i(X)$ depend only on the {\em rational} homotopy type of $X$.
\end{theorem}
\begin{proof}
The characteristic polynomial of an Iwasawa module is a pseudo-isomorphism invariant, so $X$ has the same Iwasawa modules as the wedge of finite spectra $(X_{even})_{torsfree}\vee (X_{odd})_{torsfree}$ constructed in Theorem \ref{pseudoequiv thm}. Since $X_{even}$ (respectively, $X_{odd}$) has its $p$-local homology concentrated in even (respectively, odd) degrees, the Atiyah-Hirzebruch spectral sequence implies that its $p$-adic $K$-theory is also concentrated in even (respectively, odd) degrees. 

Hence, for each $i,j$, the Iwasawa module $\epsilon_j(\KUp)^{2i}(X)$ is pseudo-isomorphic to $\epsilon_j(\KUp)^{2i}((X_{even})_{torsfree})$. Since $(X_{even})_{torsfree}$ has torsion-free $p$-local homology, it also has torsion-free $p$-adic $K$-theory by Dold's Atiyah-Hirzebruch argument (mentioned already in the proof of Theorem \ref{pseudoequiv thm}). Hence the $p$-adic $K$-theory of $(X_{even})_{torsfree}$ embeds into its rationalization. The Adams operations on the rationalized $K$-theory of a finite spectrum depend only on the rationalization of the spectrum; this is an old observation, and easily provable by the chain of isomorphisms \eqref{iso 0901} through \eqref{iso 0909} together with the fact that $H\mathbb{Q}\wedge DX$ splits as a wedge of suspensions of $H\mathbb{Q}$. 
Hence the characteristic polynomial of each Iwasawa module $\epsilon_j(\KUp)^{2i}(X)$ is determined by the rational homotopy type of $(X_{even})_{torsfree}$. 

A completely analogous argument shows that the characteristic polynomial of each Iwasawa module $\epsilon_j(\KUp)^{2i+1}(X)$ is determined by the rational homotopy type of $(X_{odd})_{torsfree}$. Hence the characteristic polynomials of the Iwasawa modules of $X$ are determined by the rational homotopy type of $(X_{even})_{torsfree}\vee (X_{odd})_{torsfree}$. The second part of Lemma \ref{pseudoequivs lemma 1} then gives us that the the characteristic polynomials of the Iwasawa modules of $X$ are determined by the rational homotopy type of $X$ itself.
\end{proof}

\begin{lemma}
\label{SIMC}
    Let $X$ be a finite spectrum such that $H_*(X;\mathbb{Z}_{(p)})$ is torsion-free and concentrated in even degrees.  
Let $a$ and $b$ be the least and greatest integers $i$, respectively, such that $H_{2i}(X;\mathbb{Q})$ is nontrivial.
Then, for each $j=0,1,\dots,p-2$, we have
    \begin{align*}
        f^{\epsilon_j(\KUp)^0(X)}((1+p)^{1-n}-1) &\sim_p \left|\pi_{2(n-1)-1} L_{K(1)}DX\right|
    \end{align*}
    for all $n \equiv 1-j \mod{p-1}$ and $n \notin (1-b,1-a]$.
\end{lemma}
\begin{proof}
    To prove this, we will need two ingredients: \Cref{SIMC for Spheres} and the Atiyah-Hirzebruch spectral sequence 
    \begin{align}
\label{ahss 10}    E_2^{s,t} \cong H^s(X, \pi_{-t}L_{K(1)}S^0) &\implies \pi_{-s-t}(DX \wedge L_{K(1)}S^0) \\
\nonumber    d_r:E_r^{s,t} &\to E_r^{s+r,t-r+1}.
    \end{align}
    
    As a consequence of \Cref{SIMC for Spheres}, we can relate the special values $f^{\epsilon_j(\KUp)^0(X)}((1+p)^{1-n}-1)$ to the product of the orders of certain bidegrees in the $E_2$-term of \eqref{ahss 10}, as follows. Since $X$ is assumed to have torsion-free $p$-local homology, the universal coefficient theorem yields 
    \begin{align*}
        H^{2i}(X; \pi_{-t}L_{K(1)}S^0) &\cong H_{2i}(X; \pi_{-t}L_{K(1)}S^0) \\
        &\cong (\pi_{-t}L_{K(1)}S^0)^{r_i},
    \end{align*}
    where $r_i = \operatorname{rank}(H_{2i}(X;\Z_{(p)}))$. So,  for each factor $f_{i,j}(T)^{r_i}$ of $f^{\epsilon_j(\KUp)^0(X)}(T)$ and each integer $n$ congruent to $1-j$ modulo $p-1$, we have
    \begin{align*}
        f_{i,j}((1+p)^{1-n}-1)^{r_i} &\sim_p \left|\pi_{2(n+i-1)-1}L_{K(1)}S^0\right|^{r_i} \\
        & \sim_p \left|E_2^{2i,1-2(n+i-1)}\right|
    \end{align*}
    Hence,
    \begin{align} \label{special values and E_2 page}
        f^{\epsilon_j(\KUp)^0(X)}((1+p)^{1-n}-1) &\sim_p \prod_{i \equiv j (p-1)} \left|E_2^{2i,1-2(n+i-1)}\right|
    \end{align}

    We still need to show that the product on the right-hand side of \eqref{special values and E_2 page} is, in fact, $\left|\pi_{2(n-1)-1} L_{K(1)}DX\right|$ when $n \notin (1-b,1-a]$.
    Because $H_*(X; \Z_{(p)})$ is concentrated in even degrees and because $E_2^{s,t} \cong H_s(X; \pi_{-t}L_{K(1)}S^0)$, the $E_2$-term of \eqref{ahss 10} is concentrated in bidegrees in which $s$ is even and $t$ is either $0$ or of the form $1-2m(p-1)$ for $m \in \Z$. For degree reasons, this means that the domain or codomain of any nonzero differential must be on the $t=0$ line. This, coupled with the fact that $H_*(X; \Z_{(p)})$ is concentrated between degrees $2a$ and $2b$, gives us that $E_2^{s,t}$ will not be hit by a nonzero differential when $s+t<2a$ or $s+t>2b$. Thus, when $m < 2a$ or $m >2b$,
    \begin{align*}
        \left| \pi_{-m}(DX \wedge L_{K(1)}S^0) \right| &= \prod_{s+t=m} \left| E_2^{s,t} \right|
    \end{align*}
    A quick calculation shows that, more specifically, when $n \equiv 1-j \mod{p-1}$ and $n \notin (1-b,1-a]$,
    \begin{align} \label{abutment of AHSS and E_2 page}
        \left| \pi_{2(n-1)-1} (DX \wedge L_{K(1)}S^0) \right| &= \prod_{i \equiv j (p-1)} \left| E_2^{2i,1-2(n+i-1)} \right|
    \end{align}

Finally, we need to know that $DX \wedge L_{K(1)}S^0$ is weakly equivalent to $L_{K(1)}DX$. This is, of course, not true for {\em arbitrary} spectra $X$, since $K(1)$-localization is not smashing. However, 
\begin{itemize}
\item we have a natural transformation $\xi: -\wedge L_{K(1)}S^0\rightarrow L_{K(1)}-$, 
\item $\xi$ is a weak equivalence when evaluated on a sphere,
\item and both the domain and codomain of $\xi$ are functors which send homotopy cofiber sequences to homotopy cofiber sequences.
\end{itemize}
Consequently $\xi$ is a weak equivalence when evaluated on any {\em finite} spectrum. Hence, since $X$ is assumed finite, \eqref{special values and E_2 page} and \eqref{abutment of AHSS and E_2 page} give us
    \begin{align*}
        f^{\epsilon_j(\KUp)^0(X)}((1+p)^{1-n}-1) &\sim_p \left|\pi_{2(n-1)-1} L_{K(1)}DX\right|.
    \end{align*}   
\end{proof}

We are now in a position to prove that, for a finite spectrum $X$, the characteristic polynomials of the $2p-2$ Iwasawa modules
$\epsilon_0(\KUp)^0(X), \dots ,\epsilon_{p-2}(\KUp)^0(X),\epsilon_0(\KUp)^{-1}(X), \dots ,\epsilon_{p-2}(\KUp)^{-1}(X)$ of $X$ determine the orders of the $K(1)$-local homotopy groups of the torsion-free replacement $\overset{\circ}{X}$ of $X$. This is the content of Theorem \ref{imc 2}. In particular, in the special case that $X$ already has $p$-local homology which is torsion-free and concentrated in even degrees---like a complex projective space---Theorem \ref{imc 2} describes the orders of the $K(1)$-local homotopy groups of $X$ itself.
\begin{theorem}[Weak form of Iwasawa Main Conjecture for spectra] \label{imc 2}
    Let $X$ be a finite spectrum.  Let $\alpha$ and $\beta$ be the least and greatest integers $i$, respectively, such that $H_{i}(X;\mathbb{Q})$ is nontrivial. Then we have
\begin{align*}
 \left|\pi_{2m-1}L_{K(1)}D\overset{\circ}{X}\right| 
  &\sim_p f^{\epsilon_{-m}(\KUp)^0(X)}((1+p)^{-m}-1) \mbox{\ \ and}\\
 \left|\pi_{2m}L_{K(1)}D\overset{\circ}{X}\right| 
  &\sim_p f^{\epsilon_{-m}(\KUp)^{-1}(X)}((1+p)^{-m}-1) 
\end{align*}
for all integers $m$ satisfying $m < \frac{-\beta}{2}$, and also for all integers $m$ satisfying $m > \frac{-\alpha}{2}$.
\end{theorem}
\begin{proof}
Corollary of Theorem \ref{pseudoequiv thm} and Lemma \ref{SIMC}.
\end{proof}
To be clear, in the statement of Theorem \ref{imc 2}, the subscript $-m$ in $\epsilon_{-m}$ is to be understood as being defined modulo $p-1$.

\begin{remark}
Here is a comment on the prospects for a {\em strong} form of the Iwasawa Main Conjecture for spectra, i.e., a theorem to the effect that the {\em algebraic} $p$-adic $L$-function given simply by the characteristic polynomial of $\epsilon_j(\KUp)^i(X)$ generates the same ideal in $\Zp[[T]]$ as some {\em analytic} $p$-adic $L$-function, presumably constructed by $p$-adically interpolating some sequence of special values of some complex-analytic $L$-function. The paper \cite{salch2023kulocal} constructed, for each finite spectrum $X$ with torsion-free homology concentrated in even degrees, a ``provisional $KU$-local zeta-function'' $\dot{\zeta}_{KU}(s,X)$. The function $\dot{\zeta}_{KU}(s,X)$ is a meromorphic function on the complex plane, and the denominators of its special values in a left-hand half-plane are proven in \cite{salch2023kulocal} to coincide with the orders of the $KU$-local stable homotopy groups of $X$. 

All the zeta-functions and $L$-functions considered in \cite{salch2023kulocal} are products of Tate twists of $L$-functions of primitive Dirichlet characters. Consequently, by Kubota--Leopoldt \cite{MR0163900}, after removing the Euler factors at $p$, the special values at negative integers can be $p$-adically interpolated to yield a $p$-adic $L$-function. One might hope that this $p$-adic $L$-function is the correct ``analytic side'' of a spectral Iwasawa Main Conjecture. We do not expect this to be true in full generality, but we find it plausible for odd {\em regular} primes $p$. This is because, given an odd regular prime $p$ and a finite spectrum $X$ with torsion-free homology concentrated in even degrees, the $p$-adic valuations of the special values of the {\em algebraic} $p$-adic $L$-function to $X$ are equal to $p$-adic valuations of the special values of the {\em analytic} $p$-adic $L$-function obtained by $p$-adic interpolation of $\dot{\zeta}_{KU}(s,X)$, as a consequence of \cite[Theorem 2.8]{salch2023kulocal} and Theorem \ref{imc 2}, above. 
\end{remark}

\section{The Iwasawa $\lambda$-invariants of finite spectra.}
\label{lambda invariant controls growth of K(1)-local homotopy groups}

\begin{comment}
Consider a tower of number fields $\dots \supseteq K_2 \supseteq K_1 \supseteq K_0$
such that, for each $n$, the Galois group $\Gal(K_n/K_0)$ is isomorphic to $\mathbb{Z}/p^n\mathbb{Z}$. Such a tower is called a {\em $\hat{\mathbb{Z}}_p$-extension}. In \cite{MR0124316}, Iwasawa constructed, for every $\hat{\mathbb{Z}}_p$-extension $K_{\bullet}$, an Iwasawa module $X(K_{\bullet})$ with the property that the $p$-Sylow subgroup of the class group of $K_n$ has order 
\[ \lambda(X(K_{\bullet}))\cdot n\ +\ \mu(X(K_{\bullet}))\cdot p^n\ +\ \nu\]
for some constant $\nu$, where $\lambda(X(K_{\bullet}))$ and $\mu(X(K_{\bullet}))$ are the $\lambda$-invariant and $\mu$-invariant of the Iwasawa module $X(K_{\bullet})$, as defined above, in \cref{The Classical Iwasawa Main Conjecture}. The point is that 
\end{comment}
In \cref{The Classical Iwasawa Main Conjecture}, we explained the precise sense in which the $\lambda$-invariant and $\mu$-invariant control the asymptotic growth rate of the $p$-part of the class number, as one moves up a suitable tower of number fields. 
\begin{comment}
The most fundamental case of a $\hat{\mathbb{Z}}_p$-extension is the case in which 
\begin{itemize}
\item $K_0$ has a primitive $p$th root of unity $\zeta_p$,
\item and $K_n = K_0(\zeta_{p^n})$ for each $n$.
\end{itemize}
For such a $\hat{\mathbb{Z}}_p$-extension, the $\mu$-invariant was conjectured by Iwasawa to be zero. This conjecture of Iwasawa was proven by Ferrero and Washington in \cite{MR0528968} when $K_0$ is an abelian extension of $\mathbb{Q}$. 
\end{comment}

Consider topological analogues of these results. We have already shown in Theorem \ref{mu-invs vanish} that the $\mu$-invariants of finite spectra are all trivial, so we focus on the question of what numerical invariant of a finite spectrum $X$ has its asymptotic growth rate described by the $\lambda$-invariants of the $2(p-1)$-tuple of Iwasawa modules \[ \epsilon_0(\KUp)^0(X), \epsilon_1(\KUp)^0(X), \dots, \epsilon_{p-2}(\KUp)^0(X),\ \epsilon_0(\KUp)^{-1}(X), \epsilon_1(\KUp)^{-1}(X), \dots, \epsilon_{p-2}(\KUp)^{-1}(X)\] associated to $X$ in \cref{tuples section}. In this section we prove that, just as the $\lambda$-invariant of the Iwasawa module of a $\hat{\mathbb{Z}}_p$-extension of number fields controls the asymptotic growth rate of the $p$-parts of the class numbers, {\em the $\lambda$-invariants of the Iwasawa modules of a finite spectrum $X$ control the asymptotic growth rate of the graded average order of the $K(1)$-local homotopy groups of $X$.} 

Our method of proof is centered on using additive invariants and algebraic $K_0$ of the category of finite spectra, since this approach is one that can potentially be generalized to higher-height analogues, using noncommutative Iwasawa theory, in the sense of \cite{MR1924402} and \cite{MR2217048}. The authors hope to pursue this idea in later work.

\subsection{Asymptotic graded averages.}
\label{Asymptotic graded averages.}

\begin{comment}
Throughout the rest of this section, we write ``bi-infinite sequence'' to mean a sequence indexed by $\mathbb{Z}$ rather than $\mathbb{N}$, i.e., a sequence of the form $\{ \dots , k_{-2}, k_{-1}, k_0, k_1, k_2, \dots\}$.

We introduce some notions which we will use to study the growth rates of {\em average values} in bi-infinite sequences. The most obvious way to do this would be to ask for asymptotic estimates of the sum $\frac{1}{n}\sum_{j=1}^n k_j$ as $n\rightarrow\infty$. 
\end{comment}
In this section we will study the average values of the orders of the $K(1)$-local homotopy groups of finite spectra. Given a finite spectrum $X$, we would like to consider the average of the first $n$ terms in the sequence
\begin{equation}\label{seq 19344} \left|\pi_{1}L_{K(1)}X\right|, \left|\pi_{2}L_{K(1)}X\right|, \left|\pi_{3}L_{K(1)}X\right|, \dots\end{equation}
However, the average $\frac{1}{n} \sum_{j=1}^n \left|\pi_{j}L_{K(1)}X\right|$ is not necessarily defined: the trouble is that the sequence \eqref{seq 19344} may include finitely many terms which are infinite. Since we are interested in asymptotics, it is harmless to skip the first $m$ terms, and consider the average $\frac{1}{n} \sum_{j=1+m}^{n+m} \left|\pi_{j}L_{K(1)}X\right|$ for $m>>0$. 

There is another point that needs explaining. It will be extremely convenient to arrange for our asymptotic growth rates to be {\em additive} invariants of spectra, i.e., if $X \rightarrow Y \rightarrow Z$ is a cofiber sequence of spectra, then the sum of the growth rates for $X$ and $Z$ ought to be the growth rate for $Y$. In the special case $Y = 0$, we see that the growth rate for $\Sigma X$ must then somehow be equal to $-1$ times the growth rate for $X$. The way to arrange this is to study the asymptotics of the {\em alternating} sum \[ \frac{1}{n}\sum_{j={1+m}}^{n+m} (-1)^j \left| \pi_jL_{K(1)}X\right| \] as $n\rightarrow\infty$, for $m>>0$. The heuristic is that we are studying the asymptotic growth rate of the average order of the $K(1)$-local stable homotopy groups of $X$, but ``average'' must be understood as a {\em graded} average, in which the even-dimensional elements are counted as positive, while the odd-dimensional elements are counted as negative.

It turns out that the graded average order of $n$ successive $K(1)$-local homotopy groups of a finite spectrum $X$ grows like $\log_p(n)/2$ times a constant. Furthermore, it turns out that this constant is determined by the $\lambda$-invariants of the Iwasawa modules of $X$. The main result establishing these facts is Theorem \ref{lambda and euler char}.

\begin{definition}
Let $\alpha$ be a nonzero real number. Given a bi-infinite sequence $\dots , k_{-1}, k_0, k_1, k_2, \dots $, we will say that its {\em graded average grows like $\alpha\cdot \log_p(n)$} if the limit 
\[ \lim_{m\rightarrow\infty} \lim_{n\rightarrow\infty} \frac{\frac{1}{n}\sum_{j=1+m}^{n+m} (-1)^j k_j}{\alpha\cdot \log_p(n)}\]
exists and is equal to $1$.
\end{definition}

\begin{proposition}\label{additivity of sn}
Let $\mathcal{C}$ be a triangulated category. Fix an object $S$ of $\mathcal{C}$, and write $\pi_i$ for the functor $[\Sigma^i S,-]: \mathcal{C}\rightarrow\Ab$. %Fix an integer $n\in \{0,1\}$, and 
\begin{comment}
Suppose that, for each object $X$ of $\mathcal{C}$, the %$n$th averaging coefficient $a_n(X)$ 
logarithmic averaging coefficient $a_1(X)$ of the bi-infinite sequence
\[ \dots, \left|\pi_{-1}X\right|, \left|\pi_{0}X\right|, \left|\pi_{1}X\right|, \left|\pi_{2}X\right|, \dots\]
is finite. Then $a_1(-): \ob\ \mathcal{C} \rightarrow\mathbb{R}$ is an additive invariant, i.e., given any cofiber sequence $X \rightarrow Y \rightarrow Z$ in $\mathcal{C}$, we have $a_1(X) + a_1(Z) = a_1(Y).$
\end{comment}
Make the following assumptions:
\begin{itemize} 
\item $X \rightarrow Y \rightarrow Z$ in $\mathcal{C}$ is a cofiber sequence in $\mathcal{C}$.
\item $\alpha_X,\alpha_Z$ are real numbers such that the graded average of $\dots ,\left|\pi_{-1}(X)\right|, \left|\pi_0(X)\right|, \left|\pi_1(X)\right|,\dots$ grows like $\alpha_X\log_p(n)$, and the graded average of $\dots ,\left|\pi_{-1}(Z)\right|, \left|\pi_0(Z)\right|, \left|\pi_1(Z)\right|,\dots$ grows like $\alpha_Z\log_p(n)$.
\item The limit $\lim_{n\rightarrow\infty} \frac{\left|\pi_n(X)\right|}{n\log_p(n)}$ exists and is equal to zero. 
\end{itemize}
Then the graded average of $\dots ,\left|\pi_{-1}(Y)\right|, \left|\pi_0(Y)\right|, \left|\pi_1(Y)\right|,\dots$ grows like $(\alpha_X+\alpha_Z)\log_p(n)$.
\end{proposition}
\begin{proof}
We have the long exact sequence
\[  \dots \stackrel{g_{n+1}}{\longrightarrow} \pi_{n+1}(Z) \stackrel{h_{n+1}}{\longrightarrow} \pi_n(X) \stackrel{f_n}{\longrightarrow} \pi_n(Y) \stackrel{g_n}{\longrightarrow} \pi_{n}(Z) \stackrel{h_{n}}{\longrightarrow} \pi_{n-1}(X)\stackrel{f_{n-1}}{\longrightarrow}\dots\]
and consequently equalities
\begin{align}
\label{eq 291} \frac{1}{n(\alpha_X + \alpha_Z)\log_p(n)} \sum_{j=1+m}^{n+m} (-1)^j \left| \pi_j(Y)\right| 
  &= \frac{1}{n(\alpha_X + \alpha_Z)\log_p(n)} \sum_{j=1+m}^{n+m} (-1)^j \left( \left| \im f_n\right| + \left| \im g_n\right|\right)\\
\nonumber  &= \frac{1}{n(\alpha_X + \alpha_Z)\log_p(n)} \sum_{j=1+m}^{n+m} (-1)^j \left( \left| \pi_j(X)\right| + \left| \pi_j(Z)\right| - \left| \im h_{j+1}\right| - \left| \im h_j\right|\right)\\
 \label{eq 293}  &= \frac{1}{n(\alpha_X + \alpha_Z)\log_p(n)} \left(  (-1)^m\left| \im h_{1+m}\right| - (-1)^{m+n}\left| \im h_{1+m+n}\right| + \sum_{j=1+m}^{n+m} (-1)^j \left( \left| \pi_j(X)\right| + \left| \pi_j(Z)\right|\right)\right).
\end{align}
Since $\im h_j$ is a subgroup of $\pi_{j-1}(X)$, the assumption that $\lim_{n\rightarrow \infty} \frac{\pi_n(X)}{n\log_p(n)} = 0$ is enough to ensure that applying $\lim_{m\rightarrow\infty}\lim_{n\rightarrow\infty}$ to \eqref{eq 293} yields
\[ \lim_{m\rightarrow\infty}\lim_{n\rightarrow\infty}\frac{1}{(\alpha_X + \alpha_Z)\log_p(n)} \sum_{j=1+m}^{n+m}\frac{1}{n} (-1)^j \left( \left| \pi_j(X)\right| + \left| \pi_j(Z)\right|\right),\]
i.e., $1$. Hence $\lim_{m\rightarrow\infty}\lim_{n\rightarrow\infty}$ applied to \eqref{eq 291} yields $1$, and hence the homotopy groups of $Y$ have graded average which grows like $(\alpha_X + \alpha_Z)\log_p(n)$.
\end{proof}
As a consequence of Proposition \ref{additivity of sn} and a simple algebraic-$K$-theoretic argument, we have:
\begin{theorem}\label{avging coeff thm}
Let $p$ be an odd prime, and let $X$ be a finite spectrum with Euler characteristic \linebreak $\chi(X) = \sum_n (-1)^n\dim_{\mathbb{Q}}H_n(X;\mathbb{Q})$. Then the \begin{comment} averaging coefficients of the $K(1)$-local stable homotopy groups of $X$ are as follows:
\begin{align*}
 a_0\left(\left|\pi_*(L_{K(1)}X)\right|\right) &= \left\{ \begin{array}{ll} 
 0\mbox{\ if\ } H_*(X;\mathbb{Q})\cong 0 \\
 \infty\mbox{\ if\ } H_*(X;\mathbb{Q})\ncong 0 \end{array} \right.\\
 a_1\left(\left|\pi_*(L_{K(1)}X)\right|\right) &= \frac{1}{2\log(p)} \chi(X)\\
\end{align*}
\end{comment}
graded average of the orders of $n$ consecutive $K(1)$-local homotopy groups of $X$ grows like $\frac{-\chi(X)}{2}\cdot \log_p(n)$.
\end{theorem}
\begin{proof}
We first prove the claim in the case that $X = S^0$. Let $s_n$ denote the graded average $\frac{1}{2(p-1)p^n} \sum_{j=1}^{2(p-1)p^n} (-1)^j \left| \pi_j L_{K(1)}S^0\right|$. By elementary algebra and the calculation of the $K(1)$-local homotopy groups of spheres, we have
\begin{align*}
 s_n 
  &= \frac{1}{2(p-1)p^n}  \sum_{i=1}^{p^n}(1- p^{\nu_p(i)+1}) \\
  &= \frac{1}{2(p-1)p^n}  \left( p^n - p \sum_{i=1}^{p^n}p^{\nu_p(i)}\right) \\
  &= \frac{1}{2(p-1)p^n}  \left( p^n - p ((n+1)p^n - np^{n-1})\right) \\
  &= \frac{1}{2(p-1)p^n}  (n+1)(p^n - p^{n+1}) \\
  &= \frac{-1-n}{2}.
\end{align*}
Consider the sequence of nonpositive rational numbers
\begin{equation}\label{seq 610}  \frac{1}{2 \log_p(2)}\sum_{j=1}^2 (-1)^j\left| \pi_j(L_{K(1)}S^0) \right| ,\ \ 
 \frac{1}{3\log_p(3)}\sum_{j=1}^3 (-1)^j\left| \pi_j(L_{K(1)}S^0) \right| ,\ \ 
 \frac{1}{4\log_p(4)}\sum_{j=1}^4 (-1)^j\left| \pi_j(L_{K(1)}S^0) \right| ,\dots .\end{equation}
The least (i.e., highest absolute value) terms in \eqref{seq 610} are taken in the subsequence
\begin{equation}
\label{seq 611} 
\begin{multlined}\frac{1}{(2(p-1)p^0-1)\log_p(2(p-1)p^0-1)}\sum_{j=1}^{2(p-1)p^0-1} (-1)^j\left| \pi_j(L_{K(1)}S^0) \right| ,\\
\ \ \ \ \ \frac{1}{(2(p-1)p^1-1)\log_p(2(p-1)p^1-1)}\sum_{j=1}^{2(p-1)p^1-1} (-1)^j\left| \pi_j(L_{K(1)}S^0) \right|,\\
\ \ \ \ \ \frac{1}{(2(p-1)p^2-1)\log_p(2(p-1)p^2-1)}\sum_{j=1}^{2(p-1)p^2-1} (-1)^j\left| \pi_j(L_{K(1)}S^0) \right|,\\
\ \ \ \ \ \frac{1}{(2(p-1)p^3-1)\log_p(2(p-1)p^3-1)}\sum_{j=1}^{2(p-1)p^3-1} (-1)^j\left| \pi_j(L_{K(1)}S^0) \right|, \dots . \end{multlined}
 \end{equation}
Write $t_n$ for the $n$th term, counting from zero, in the sequence \eqref{seq 611}. Then we have
\begin{align*} t_n &= \frac{2(p-1)p^ns_n - 1}{(2(p-1)p^n-1)\log_p(2(p-1)p^n-1)},\end{align*}
and consequently
\begin{align*}
 \lim_{n\rightarrow\infty} t_n 
%  &= \lim_{n\rightarrow\infty} \frac{(2(p-1)p^n-1)s_n + s_n - 1}{(2(p-1)p^n-1)\log_p(2(p-1)p^n-1)}\\
  &= \lim_{n\rightarrow\infty} \frac{(2(p-1)p^n-1)(-1-n)/2 + (-1-n)/2 - 1}{(2(p-1)p^n-1)\log_p(2(p-1)p^n-1)}\\
  &= \lim_{n\rightarrow\infty} \frac{(-1-n)/2}{\log_p(2(p-1)p^n-1)}\\
  &= \frac{-1}{2}.
\end{align*}
Consequently the graded average of the orders of $n$ consecutive $K(1)$-local homotopy groups of spheres grows no slower than $\frac{-\log_p(n)}{2}$.
To bound the growth from above, consider the subsequence of \eqref{seq 610} given by 
\begin{equation}\label{seq 612}
\begin{multlined} \frac{1}{(2(p-1)p^0-2)\log_p(2(p-1)p^0-2)}\sum_{j=1}^{2(p-1)p^0-2} (-1)^j\left| \pi_j(L_{K(1)}S^0) \right| ,\\
\frac{1}{(2(p-1)p^1-2)\log_p(2(p-1)p^1-2)}\sum_{j=1}^{2(p-1)p^1-2} (-1)^j\left| \pi_j(L_{K(1)}S^0) \right|,\\
\frac{1}{(2(p-1)p^2-2)\log_p(2(p-1)p^2-2)}\sum_{j=1}^{2(p-1)p^2-2} (-1)^j\left| \pi_j(L_{K(1)}S^0) \right|,\\
\frac{1}{(2(p-1)p^3-2)\log_p(2(p-1)p^3-2)}\sum_{j=1}^{2(p-1)p^3-2} (-1)^j\left| \pi_j(L_{K(1)}S^0) \right|, \dots 
 .\end{multlined}\end{equation}
The greatest (i.e., least absolute value) terms in \eqref{seq 610} are taken in the subsequence \eqref{seq 612}. Write $u_n$ for the $n$th term in \eqref{seq 612}, counting from zero.
Then we have 
\begin{align*} u_n &= \frac{2(p-1)p^ns_n - 1 + p^{n+1}}{(2(p-1)p^n-2)\log_p(2(p-1)p^n-2)},\end{align*}
and consequently
\begin{align*}
 \lim_{n\rightarrow\infty} u_n 
%  &= \lim_{n\rightarrow\infty} \frac{(2(p-1)p^n-2)s_n + 2s_n - 1 + p^{n+1}}{(2(p-1)p^n-2)\log_p(2(p-1)p^n-2)}\\
  &= \lim_{n\rightarrow\infty} \frac{(2(p-1)p^n-2)(-1-n)/2  - n - 2 + p^{n+1}}{(2(p-1)p^n-2)\log_p(2(p-1)p^n-2)}\\
  &= \lim_{n\rightarrow\infty} \frac{-n/2}{\log_p(2(p-1)p^n-2)}\\
  &= \frac{-1}{2}.
\end{align*}
Consequently \eqref{seq 610} converges to $-1/2$, i.e., the graded average of the orders of $n$ consecutive $K(1)$-local homotopy groups of spheres grows like $\frac{-\log_p(n)}{2}$.

In the sequence of positive rational numbers $\left|\pi_2(L_{K(1)}S^0)\right|/(2\log_p(2)), \left|\pi_3(L_{K(1)}S^0)\right|/(3\log_p(3)),\left|\pi_4(L_{K(1)}S^0)\right|/(4\log_p(4)) , \dots$, the highest values are taken at terms of the form 
\begin{align*}
 \frac{\left|\pi_{2p^n(p-1)-1}(L_{K(1)}S^0)\right|}{(2p^n(p-1)-1)\log_p(2p^n(p-1)-1)} &= \frac{p^{n+1}}{(2p^n(p-1)-1)\log_p(2p^n(p-1)-1)}. 
\end{align*}
As $n\rightarrow \infty$, even these highest values go to zero. Hence $\lim_{n\rightarrow\infty}\left|\pi_n(L_{K(1)}S^0)\right|/(n\log_p(n)) = 0$. By induction on cells using Proposition \ref{additivity of sn}, we have a well-defined {\em additive} invariant of finite spectra, given by sending a finite spectrum $X$ to the unique real number $\alpha(X)$ such that the graded average of the orders of $n$ consecutive $K(1)$-local homotopy groups of $X$ grows like $\alpha(X) \cdot\log_p(n)$. 

Since $X \mapsto \alpha(X)$ is an additive invariant of finite spectra, it depends only on the value of the {\em universal} additive invariant of finite spectra, i.e., the Euler characteristic. Hence we can calculate the value of $\alpha$ on any finite spectrum $X$ by multiplying the Euler characteristic $\chi(X)$ by the value of $\alpha$ on the generator $S^0$ of $K_0(\Fin\Sp)$. The calculation $\alpha(S^0) = -1/2$, above, yields the claim in the statement of the theorem.
\end{proof}
\begin{comment}
In the homotopy groups of $L_{K(1)}S^0$, the first element of order $p^i$ occurs in $\pi_{2(p-1)p^{i-1}-1}L_{K(1)}S^0$, so it is the value of the alternating sum $\sum_{j=m}^n (-1)^j\left|\pi_j(L_{K(1)}S^0)\right|$ in the cases $n=2(p-1)p^{i-1}-1$ which controls the averaging coefficients of $\left|\pi_*(L_{K(1)}S^0)\right|$. 
We have
\begin{align*}
\lim_{i\rightarrow \infty} \frac{\sum_{j=1}^{2(p-1)p^{i-1}-1} (-1)^j \left|\pi_j(L_{K(1)}S^0)\right|}{2(p-1)p^{i-1}-1} 
 &= \lim_{i\rightarrow \infty} \frac{ip^i - (i-1)p^{i-1}}{2(p-1)p^{i-1}-1} \\
 &= \lim_{i\rightarrow \infty} \frac{i}{2},\mbox{\ i.e.,} \\
\lim_{n\rightarrow \infty} \frac{\sum_{j=1}^{n} (-1)^j \left|\pi_j(L_{K(1)}S^0)\right|}{n} \\
 &= \lim_{n\rightarrow \infty} \frac{1}{2}\left( 1+\log_p\left( \frac{1 + n}{2p-2}\right)\right) \\
 &= \lim_{n\rightarrow \infty} \frac{\log(1+n)}{2\log(p)}.
\end{align*}
so the constant averaging coefficient is infinite, and the logarithmic averaging coefficient is $\frac{1}{2\log(p)}$.
\end{comment}

\subsection{The $\lambda$-invariants of a finite spectrum}
\label{Iwasawa invariants of finite spectra}

\begin{definition}\label{def of tot lambda}
Let $X$ be a finite spectrum.  
By $\lambda(X)$, the {\em total $\lambda$-invariant of $X$,} we mean the sum of differences
\[ \sum_{j=0}^{p-2}\left( \lambda(\epsilon_j(\KUp)^0(X)) - \lambda(\epsilon_j(\KUp)^{-1}(X))\right) \]
of the $\lambda$-invariants of the Iwasawa modules associated to $X$. 
\end{definition}

\begin{lemma}\label{additivity of lambda}
The total $\lambda$-invariant is equal to the Euler characteristic. Consequently the total $\lambda$-invariant is an additive invariant of finite spectra.
\end{lemma}
\begin{proof}
It is probably possible to give an intrinsic proof, based on formal properties of the characteristic polynomial and not reliant on the fact, established by Theorem \ref{iwasawa invs and rational type}, that the characteristic polynomial of each Iwasawa module of a finite spectrum $X$ depends only on the rational homotopy type of $X$. However, the shortest and simplest proof certainly uses that fact, in order to recognize that the $\lambda$-invariant of $\epsilon_j(\KUp)^n(X)$ is equal to the $\lambda$-invariant of $\epsilon_j(\KUp)^0(Y)$ where $Y$ is a wedge of spheres with the same rational homology as $X$. In \cref{KUp of the Spheres}, we calculated the characteristic polynomials and $\lambda$-invariants of the Iwasawa modules $\epsilon_j(\KUp)^0(S^{2i})$. In particular, those calculations yield that
\begin{align*}
 \lambda(\epsilon_j(\KUp)^0(S^{2i})) &= \lambda(\epsilon_j(\KUp)^{-1}(S^{2i-1})) \\ 
  &= \left\{ \begin{array}{ll} 1 &\mbox{\ if\ } j\equiv i \mod p-1 \\ 0 &\mbox{\ otherwise}.\end{array}\right.\\
 \lambda(\epsilon_j(\KUp)^{-1}(S^{2i})) &= \lambda(\epsilon_j(\KUp)^{0}(S^{2i-1})) \\ 
  &= 0.
\end{align*}
From this, and the straightforward observation that $\lambda(\epsilon_j(\KUp)^n(X\vee X^{\prime})) = \lambda(\epsilon_j(\KUp)^n(X)) + \lambda(\epsilon_j(\KUp)^n(X^{\prime}))$ for all finite spectra $X,X^{\prime}$, we have the formula
\begin{align*}
 \lambda(\epsilon_j(\KUp)^0(X)) 
  &= \sum_{n\equiv j \mod p-1} \dim_{\mathbb{Q}}H_{2n}(X) \\
 \lambda(\epsilon_j(\KUp)^{-1}(X)) 
  &= \sum_{n\equiv j \mod p-1} \dim_{\mathbb{Q}}H_{2n-1}(X) ,
\end{align*}
and consequently the total $\lambda$-invariant is equal to the Euler characteristic.
\end{proof}

Theorem \ref{avging coeff thm} and Lemma \ref{additivity of lambda} then jointly imply:
\begin{theorem}\label{lambda and euler char}
Let $X$ be a finite spectrum. Then the graded average of the orders of $n$ consecutive $K(1)$-local homotopy groups of $X$ grows like $\frac{-\lambda(X)}{2} \cdot\log_p(n)$.
\end{theorem}

\bibliographystyle{plain}
\bibliography{project_refs}

\end{document}